\documentclass[11pt]{article}

\usepackage{amssymb,amsmath,amsthm, enumerate}
\usepackage[utf8]{inputenc}
\usepackage{epsfig}
\usepackage{breqn}

\usepackage{rotating}
\usepackage{amsfonts}
\usepackage{setspace}
\usepackage{fullpage}
\usepackage{enumitem}
\usepackage{bbold} 
\usepackage{amsmath}
\usepackage{comment}
\usepackage{pgf,tikz}
\usepackage{mathtools}  
\usepackage{hyperref}
\usepackage{subcaption}
\usepackage{graphicx}
\usepackage{lipsum}
\usepackage{authblk}

\usepackage{setspace}

\theoremstyle{definition}

\newtheorem{theorem}{Theorem}[section]
\newtheorem{lemma}[theorem]{Lemma}

\newtheorem{definition}[theorem]{Definition}

\newtheorem{conjecture}[theorem]{Conjecture}

\newcommand{\floor}[1]{\left\lfloor{#1}\right\rfloor}

\newcommand{\pa}[1]{\mathcal{P}_{#1}}

\newcommand{\abs}[1]{\left\lvert{#1}\right\rvert}

\def\path{semi-path}

\begin{document}
\title{The Maximum Wiener Index of Maximal Planar Graphs}

\author[1]{Debarun Ghosh}
\author[1,2]{Ervin Gy\H{o}ri} 
\author[1,3]{Addisu Paulos}
\author[1,2]{Nika Salia}
\author[1,4]{Oscar Zamora} 

\affil[1]{Central European University, Budapest.\par
 \texttt{ghosh_debarun@phd.ceu.edu, gyori.ervin@renyi.hu, paulos_addisu@phd.ceu.edu, salia.nika@renyi.hu, zamora-luna_oscar@phd.ceu.edu}}
 \affil[2]{Alfr\'ed R\'enyi Institute of Mathematics, Hungarian Academy of Sciences.}
\affil[3]{Addis Ababa University, Addis Ababa.}
 \affil[4]{Universidad de Costa Rica, San Jos\'e.}
\maketitle

\begin{abstract}

The Wiener index of a connected graph is the sum of the distances between all pairs of vertices in the graph. 
It was conjectured that the Wiener index of an $n$-vertex maximal planar graph is at most $\lfloor\frac{1}{18}(n^3+3n^2)\rfloor$.
We prove this conjecture and for every $n$, $n \geq 10$,  determine the unique $n$-vertex maximal planar graph for which this maximum is attained.

\end{abstract}
\section{Introduction}

The Wiener index is a graph invariant based on distances in the graph. For a connected graph $G$, the Wiener index is the sum of distances between all unordered pairs of vertices in the graph and is denoted by $W(G)$. That means,
\begin{equation*}
    W(G) = \sum_{\{u,v\} \subseteq V(G)} d_G(u,v).
\end{equation*} 
Where  $d_G(u,v)$ denotes the distance from $u$ to $v$ i.e. the minimum length of a path from $u$ to $v$ in the graph $G$. 

It was first introduced by Harry Wiener in 1947, while studying its correlations with boiling points of paraffin considering its molecular structure  \cite{first}. Since then, it has been one of the most frequently used topological indices in chemistry, as molecular structures are usually modelled as undirected graphs. Many results on the Wiener index and closely related parameters such as the gross status \cite{third}, the distance of graphs \cite{fourth} and the transmission \cite{fifth} have been studied. A great deal of knowledge on the Wiener index is accumulated in several survey papers \cite{1,2,3,4,5}.
Finding a sharp bound on the Wiener index for graphs under some constraints, has been one of the  research topics attracting many researchers. 

The most basic upper bound of $W(G)$ states that, if $G$ is a connected graph of order $n$,
then
\begin{equation}
    W(G) \leq \frac{(n-1)n(n+1)}{6},
\end{equation} 
which is attained only by a path \cite{15,7,8}.
Many sharp or asymptotically sharp bounds on $W(G)$ in terms of other graph parameters are known, for instance, minimum degree \cite{9,10,11}, connectivity \cite{12,13}, edge-connectivity \cite{14,15} and maximum degree \cite{16}. For finding more details in mathematical aspect of Wiener index, see also results \cite{17,18,19,20,21,22,23,24,25,26}.

One can study the Wiener index of the family of connected planar graphs. 
Since the bound given in (1) is attained by a path, it is natural to ask for some particular family of planar graphs. 
For instance, the family of maximal planar graphs. 
The Wiener index of maximal planar graph with $n$ vertices, $n\geq 3$ has a sharp lower bound $(n-2)^2+2$, the bound is attained by any maximal planar graph such that the distance between any pair of vertices is at most 2 (for instance a planar graph containing the $n$-vertex star).
 Z. Che and K.L. Collins \cite{28}, and independently \' E. Czabarka, P. Dankelmann, T. Olsen and L.A. Sz\' ekely \cite{29},  gave a sharp upper bound of a particular class of maximal planar graphs known as \textit{Apollonian networks}. 
An Apollonian network may be formed, starting from a single triangle embedded on the plane, by repeatedly selecting a triangular face of the embedding, adding a new vertex inside the face, and connecting the new vertex to each three vertices of the face. They showed that 
\begin{theorem}(\cite{28, 29})
Let $G$ be an Apollonian network of order $n\geq 3$. Then $W(G)$ has a sharp upper bound
\begin{equation*}W(G)\leq \bigg\lfloor\frac{1}{18}(n^3+3n^2)\bigg\rfloor=
\begin{cases}
\frac{1}{18}(n^3+3n^2), &\text{if $n\equiv 0(mod \ 3)$;}\\
\frac{1}{18}(n^3+3n^2-4), &\text{if $n\equiv 1(mod \ 3)$;}\\
\frac{1}{18}(n^3+3n^2-2), &\text{if $n\equiv 2(mod \ 3)$.}
\end{cases}
\end{equation*}
\end{theorem}
It has been shown explicitly  that the Wiener index is attained for the maximal planar graphs $T_n$, we will give the construction of $T_n$ in the next section, see Definition \ref{xx}.  
The authors in \cite{28} also conjectured that this bound also holds for every maximal planar graph. 
It has been shown that the conjectured bound holds asymptotically \cite{29}. 
In particular they showed the following result.
\begin{theorem}(\cite{29})\label{az}
Let $k\in\{3,4,5\}$. 
Then there exists a constant $C$ such that
\begin{equation*}
    W(G)\leq \frac{1}{6k}n^3+C n^{5/2}
\end{equation*}
for every $k$-connected maximal planar graph of order $n$.
\end{theorem}
In this paper, we confirm the conjecture. 
\begin{theorem}\label{Main_Theorem}
Let $G$ be an $n$, $n\geq 6$, vertex, maximal, planar graph. Then we have,

\begin{equation*}W(G)\leq \bigg\lfloor\frac{1}{18}(n^3+3n^2)\bigg\rfloor=
\begin{cases}
\frac{1}{18}(n^3+3n^2), &\text{if $n\equiv 0(mod \ 3)$;}\\
\frac{1}{18}(n^3+3n^2-4), &\text{if $n\equiv 1(mod \ 3)$;}\\
\frac{1}{18}(n^3+3n^2-2), &\text{if $n\equiv 2(mod \ 3)$.}
\end{cases}
\end{equation*}

Equality holds if and only if $G$ is isomorphic to $T_n$ for all $n$, $n\geq 10$.
\end{theorem}
\section{Notations and Preliminaries}
Let $G$ be a graph. We denote vertex set and edge set of $G$ by $V(G)$ and $E(G)$ respectively.
 For a vertex set  $S \subset V(G)$, the \textit{status} of  $S$ is defined as the sum of all distances from $S$ to all vertices of the graph. It is denoted by $\sigma_G(S)$, thus
\begin{align*}
    \sigma_G(S)=\sum_{u\in V(G)}d_G(S,u).
\end{align*}
For simplicity, we may use the notation $\sigma(S)$ instead of $\sigma_G(S)$ when the underlined graph is clear.
We have,
\begin{equation*}
    W(G) = \frac{1}{2} \sum_{v \in V(G)} \sigma_G(v).
\end{equation*}

Here we are defining an Apollonian network $T_n$ on $n$ vertices. We will prove later that it is the unique  maximal planar graph which maximizes the Wiener index.
\begin{definition}\label{xx}
The Apollonian network $T_n$ is the maximal planar graph on $n\geq 3$ vertices, with the following structure, see Figure \ref{apollonian}.

If $n$ is a multiple of $3$, then the  vertex set of $T_n$ can be  partitioned in three sets of same size, $A=\{a_1,a_2,\cdots,a_k\}$, $B=\{b_1, b_2,\dots,b_k\}$ and $C=\{c_1,c_2,\cdots,c_k\}$. The edge set of $T_n$ is the union of following three sets
$E_1=\bigcup_{i=1}^{k}\{(a_i,b_i), (b_i,c_i), (c_i,a_i)\}$ forming concentric triangles,
$E_2=\bigcup_{i=1}^{k-1} \{(a_i,b_{i+1}), (a_i,c_{i+1}), (b_i,c_{i+1})\}$ forming `diagonal' edges, and $E_3= \bigcup_{1}^{k-1} \{(a_i,a_{i+1}),\\ (b_i,b_{i+1}), (c_i,c_{i+1})\}$ forming paths in each vertex class, see Figure \ref{a}. 
Note, that there are two triangular faces $a_1,b_1,c_1$ and $a_k,b_k,c_k$.

If $3|n-1$, then $T_n$ is the Apollonian network which may be obtained from $T_{n-1}$ by adding a degree three vertex in the face $a_1,b_1,c_1$ or $a_{\frac{n-1}{3}},b_{\frac{n-1}{3}},c_{\frac{n-1}{3}}$, see Figure \ref{b}.  Note that both graphs are isomorphic.

If $3|n-2$, then $T_n$ is the Apollonian network which may be obtained from $T_{n-2}$ by adding a degree three vertex in each of the faces $a_1,b_1,c_1$ and $a_{\frac{n-1}{3}},b_{\frac{n-1}{3}},c_{\frac{n-1}{3}}$, see Figure \ref{c}.
\end{definition}

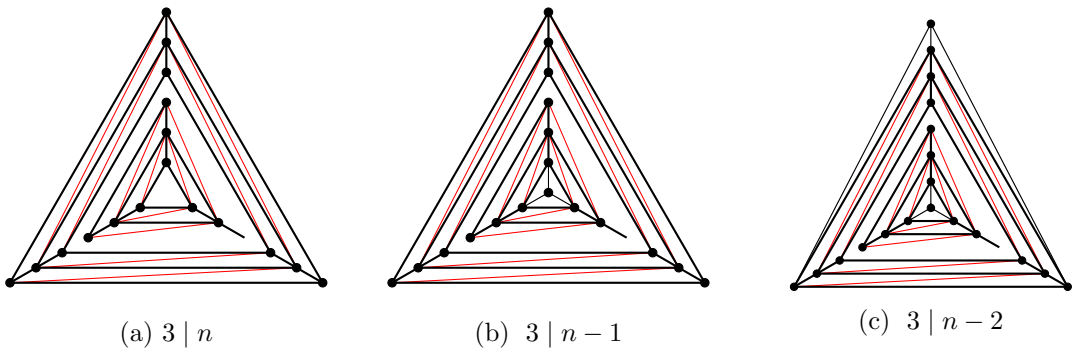
\begin{figure}[h]
\begin{subfigure}{0.3\linewidth}
\centering
\begin{tikzpicture}[scale=.4]
\foreach \x in {1,2,4,5}{
\draw[thick]  (90:\x) -- (90:{\x+1}) (-30:\x) -- (-30:{\x+1})  (210:\x) -- (210:{\x+1});
\draw[red] (90:{\x+1}) --  (-30:\x) -- (210:{\x+1})  (210:\x) -- (90:{\x+1});
}
\foreach \x in {1,2,4,5,...,6}{
\filldraw (90:\x) circle (4pt)  (-30:\x) circle (4pt) (210:\x) circle (4pt);
\draw[thick] (90:\x) --  (-30:\x) -- (210:\x) -- (90:\x);
}
\filldraw (90:3) circle (4pt) (210:3)  circle (4pt);
\draw[thick] (90:3) -- (210:3);
 \draw (0,-3) node[below]{};
\end{tikzpicture}
\caption{$3\mid n$}
\label{a}
\end{subfigure}
\begin{subfigure}{0.3\linewidth}
\centering
\begin{tikzpicture}[scale=.4]
\filldraw (210:1) -- (0,0) circle (4pt) -- (90:1) (-30:1) -- (0,0);
\foreach \x in {1,2,4,5}{
\draw[thick]  (90:\x) -- (90:{\x+1}) (-30:\x) -- (-30:{\x+1})  (210:\x) -- (210:{\x+1});
\draw[red] (90:{\x+1}) --  (-30:\x) -- (210:{\x+1})  (210:\x) -- (90:{\x+1});
}
\foreach \x in {1,2,4,5,...,6}{
\filldraw (90:\x) circle (4pt)  (-30:\x) circle (4pt) (210:\x) circle (4pt);
\draw[thick] (90:\x) --  (-30:\x) -- (210:\x) -- (90:\x);
}
\filldraw (90:3) circle (4pt) (210:3)  circle (4pt);
\draw[thick] (90:3) -- (210:3);
 \draw (0,-3) node[below]{};
\end{tikzpicture}
\caption{$\; 3\mid n-1$}
\label{b}
\end{subfigure}
\begin{subfigure}{0.3\linewidth}
\centering
\begin{tikzpicture}[scale=.35]
\filldraw (210:6) -- (0,7) circle (4pt) -- (-30:6) (0,7) -- (0,6);
\filldraw (210:1) -- (0,0) circle (4pt) -- (90:1) (-30:1) -- (0,0);
\foreach \x in {1,2,4,5}{
\draw[thick]  (90:\x) -- (90:{\x+1}) (-30:\x) -- (-30:{\x+1})  (210:\x) -- (210:{\x+1});
\draw[red] (90:{\x+1}) --  (-30:\x) -- (210:{\x+1})  (210:\x) -- (90:{\x+1});
}
\foreach \x in {1,2,4,5,...,6}{

\filldraw (90:\x) circle (4pt)  (-30:\x) circle (4pt) (210:\x) circle (4pt);
\draw[thick] (90:\x) --  (-30:\x) -- (210:\x) -- (90:\x);
}
\filldraw (90:3) circle (4pt) (210:3)  circle (4pt);
\draw[thick] (90:3) -- (210:3);
\end{tikzpicture}
\caption{$ \; 3 \mid n-2$}
\label{c}
\end{subfigure}
\caption{Apollonian networks maximizing Wiener index of maximal planar graphs}
\label{apollonian}
\end{figure}

The following lemmas will be used in the proof of Theorem \ref{Main_Theorem}.

\begin{lemma}\label{S_Connected_Cycle_lemma}
Let $G$ be a $s$-connected, maximal planar graph. Then every cut set of size $s$ contains a Hamiltonian cycle of length $s$.
\end{lemma}
\begin{proof}
Let us assume that a cut set of size $s$ is $S=\{v_1,v_2,\dots,v_s\}$. Let $u$ and $w$ be two vertices such that any path from $u$ to $w$ contains at least one vertex from $S$. 
Since $G$ is $s$-connected, by Menger's Theorem, there are $s$ vertex disjoint paths from $u$ to $w$. 
Each of the paths intersects $S$ in disjoint nonempty sets, therefore each of the paths contain exactly one vertex from $S$. 
We may assume, that in a particular planar embedding of $G$, those paths are ordered in such a way that one of the two regions determined by the cycle obtained from two paths from $u$ to $w$  containing $v_{i_x}$ and ${v_{i_{x+1}}}$ has no vertex from $S$  (where indices are taken modulo $s$), see Figure \ref{S_connected_Picture}. 
Then we must have the edges $\{v_{i_x},v_{i_{x+1}}\}$, otherwise, from the maximality of the planar graph, there is a path from the vertex $u$ to the vertex $w$ that does not contain a vertex from $S$, a contradiction. 
Therefore we have a cycle of length $s$ on the vertex set $S$, $v_{i_{1}},v_{i_{2}},\cdots,v_{i_{s}},v_{i_{1}}$.
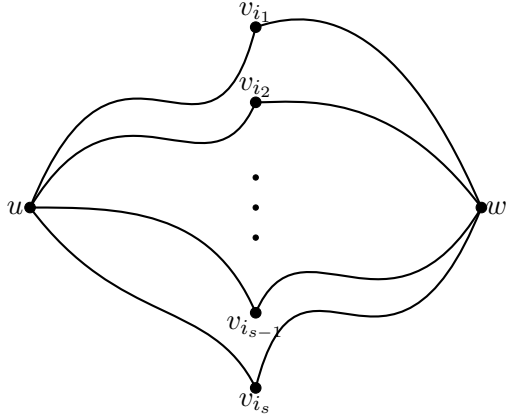
\begin{figure}[h]
\centering
\begin{tikzpicture}[scale=0.2]
\draw[fill=black](-15,0)circle(10pt);
\draw[fill=black](15,0)circle(10pt);
\draw[fill=black](0,12)circle(10pt);
\draw[fill=black](0,-12)circle(10pt);
\draw[fill=black](0,7)circle(10pt);
\draw[fill=black](0,-7)circle(10pt);
\draw[fill=black](0,2)circle(5pt);
\draw[fill=black](0,-2)circle(5pt);
\draw[fill=black](0,0)circle(5pt);
\draw[black,thick](-15,0)..controls (-9,15) and (-3,0) .. (0,12);
\draw[black,thick](-15,0)..controls (-9,10) and (-3,0) .. (0,7);
\draw[black,thick](-15,0)..controls (-9, -8) and (-3,-6) .. (0,-12);
\draw[black,thick](-15,0)..controls (-9,0) and (-3,0) .. (0,-7);
\draw[black,thick](15,0)..controls (9,-15) and (3,0) .. (0,-12);
\draw[black,thick](15,0)..controls (9,-10) and (3,0) .. (0,-7);
\draw[black,thick](15,0)..controls (9,8) and (3,7) .. (0,7);
\draw[black,thick](15,0)..controls (9,14) and (3,13) .. (0,12);
\node at (-16,0) {$u$};
\node at (16,0) {$w$};
\node at (0,13) {$v_{i_1}$};
\node at (0,8) {$v_{i_2}$};
\node at (0,-8) {$v_{i_{s-1}}$};
\node at (0,-13) {$v_{i_s}$};
\end{tikzpicture}
\caption{$s$ pairwise disjoint paths from $u$ to $w$.}
\label{S_connected_Picture}
\end{figure}
\end{proof}

The following definition would be particularly helpful. Given a set $S\subseteq V$, we define the Breadth First Search partition of $V$ with root $S$,  $\pa{S}^G$ or simply $\pa{S}$ when the underline graph is clear, by $\pa{S} = \{S_0,S_1,\dots\}$, where $S_0 = S$, and for $i \geq 1$, $S_i$ is the set of vertices at distance exactly $i$ from $S$. 
We refer to those sets as \emph{levels} (of $\pa{s}$), $S_1$ is the \emph{first level}, and if $k$ is the largest integer such that $S_k \neq \emptyset$, the we refer to $S_k$ as the \emph{last level}. 
We refer to $S_0$ and the last level as \emph{terminal level}.
Note that by definition every level beside the first and last is a cut set of $G$. We denote by $\pa{v}$ the Breadth First Search partition from $v$, that is the partition $\pa{\{v\}}.$ 


\begin{lemma}\label{Three_Connected_Lemma}
Let $G$ be an $n+s$ vertex graph and $S$, $S\subset V(G)$, be a set of vertices of size $s$. Such that each non-terminal level of $\pa{S}$ has size  at least $3$. Then  we have
\begin{equation*}
\sigma(S)\leq \sigma_3(n):= \begin{cases}
\frac{1}{6}(n^2+3n), &\text{if $n\equiv 0 \; (mod \ 3)$;}\\
\frac{1}{6}(n^2+3n+2), &\text{if $n\equiv 1,2 \; (mod \ 3)$.}
\end{cases}
\end{equation*}
\end{lemma}

\begin{proof}

If $\pa{S} = \{S_0,S_1,\dots\}$, by definition, we have that $\sigma(S) = \displaystyle\sum_{i} i\abs{S}.$ 
Therefore 
\begin{align*}
\sigma(S) &= \abs{S_1} + 2\abs{S_2} + 3\abs{S_3} + \cdots \\&\leq 3\bigg(1+2+\dots+ \floor{\frac{n}{3}}\bigg)+\bigg(\floor{\frac{n}{3}}+1 \bigg)\bigg(n-3 \floor{\frac{n}{3}}\bigg)= \sigma_3(n). \tag*{\qedhere}
\end{align*}

\end{proof}

Similarly we can prove the following Lemmas.

\begin{lemma}\label{Four_Connected_Lemma}
Let $G$ be an $n+s$ vertex graph and $S$, $S\subset V(G)$, be a set of vertices of size $s$. Such that each non terminal level of $\pa{S}$ has size  at least $4$. Then we have
\begin{equation*}
\sigma(S)\leq \sigma_4(n):= \begin{cases}
\frac{1}{8}(n^2+4n), &\text{if $n\equiv 0 \;(mod \ 4)$;}\\
\frac{1}{8}(n^2+4n+3), &\text{if $n\equiv 1,3 \; (mod \ 4)$;}\\
\frac{1}{8}(n^2+4n+4), &\text{if $n\equiv 2 \;(mod \ 4)$.}
\end{cases}
\end{equation*}
\end{lemma}



\begin{lemma}\label{Five_Connected_Lemma}
Let $G$ be an $n+s$ vertex graph and $S$, $S\subset V(G)$, be a set of vertices of size $s$. Such that each non terminal level of $\pa{S}$ has size  at least $5$. Then we have
\begin{align*}
\sigma(S)\leq \sigma_5(n):= \begin{cases}
\frac{1}{10}(n^2+5n), &\text{if $n\equiv 0 \;(mod \ 5)$;}\\
\frac{1}{10}(n^2+5n+4), &\text{if $n\equiv 1,4\; (mod \ 5)$;}\\
\frac{1}{10}(n^2+5n+6), &\text{if $n\equiv 2, 3\; (mod \ 5)$.} 
\end{cases}
\end{align*}
\end{lemma}


\section[Proof of Theorem 1.3]{Proof of Theorem \ref{Main_Theorem}}

\begin{proof}
We are going to prove Theorem \ref{Main_Theorem} by induction, on the number of vertices. In \cite{29} it is shown that Theorem \ref{Main_Theorem} holds, for $n \leq 18$. 
Therefore, we may assume $n\geq 19$. 
Let $G$ be a maximal planar graph. The proof contains three cases depending on the connectivity of the graph $G$. 
Since $G$ is a maximal planar graph, it is either $3$, $4$ or $5$ connected. 
Thus, we may consider three different cases.

\noindent\begin{bf} Case $1$. Let $G$ be a $5$-connected graph.\end{bf}
For every fixed vertex $v \in V(G)$, consider $\pa{v}$. Since $G$ is $5$-connected, and each of the non-terminal levels  of $\pa{v}$ is a cut set, we have that each non-terminal level has size at least $5$. Therefore from 
Lemma \ref{Five_Connected_Lemma}, we have,  $$W(G)=\frac{1}{2}\sum_{v\in V(G)}\sigma(v)\leq  \frac{n}{2}     \sigma_5(n-1) \leq \frac{n}{20}(n^2+3n+2)< \bigg\lfloor\frac{1}{18}(n^3+3n^2)\bigg\rfloor,$$ for all $n\geq 4$.
Therefore we are done if $G$ is $5$-connected, since $n\geq 19$.

\noindent\begin{bf} Case $2$. Let $G$ be $4$-connected and not $5$-connected. \end{bf} Then $G$ contains a cut set of size $4$, which induces a cycle of length four, by Lemma \ref{S_Connected_Cycle_lemma}.  
Let us denote the vertices of this cut set as $v_1,v_2,v_3$ and $v_4$, forming the cycle in this given order. 
The cut set divides the plane into two regions, we will call them the inner and the outer regions respectively.  
Let us denote the number of vertices in the inner region by $x$. Let us assume, without loss of generality, that $x$ is minimal possible, but greater then one. Obviously $x\leq \frac{n-4}{2}$ or $x=n-5$.
From here on, we deal with several sub-cases depending on the value of $x$.

\begin{bf}Case $2.1$ \end{bf}In this case we assume $x\geq 4$ and $x\neq n-5$. 
Let us consider the sub-graph of $G$, say $G'$, obtained by deleting all vertices from the outer region of the cycle $v_1,v_2,v_3,v_4$ in $G$. 
The graph $G'$ is not maximal, since the outer face is a $4$-cycle. 
The graph $G$ is $4$-connected, therefore it does not contain the edges $\{v_1,v_3\}$ and $\{v_2,v_4\}$, consequently we may add any of them to $G'$, to obtain a maximal planar graph.
Adding an edge decreases the Wiener index of $G'$. 
In the following paragraph, we prove that one of the edges decrease the Wiener index of $G'$ by at most  $\frac{x^2}{16}$.

Let $A_i=\{v\in V(G')|d(v,v_i)<d(v,v_j), \forall j\in\{1,2,3,4\}\setminus\{i\}\}$ for $i\in\{1,2,3,4\}$. 
Let $A$ be the subset of vertices of $G'$ not contained in any of the $A_i$'s. 
So $A,A_1,A_2,A_3,A_4$ is a partition of vertices of $G'$. 
It is simple to observe that, adding the edge $\{v_i,v_{i+2}\}$, for $i\in \{1,2\}$, decreases the distance between a pair of vertices, then these vertices must be from $A_{i}$ and $A_{i+2}$.
If there is a vertex $u$ which has three neighbours from the cut set, without loss of generality say $v_1,v_2,v_3$, then $A_{2}=\emptyset$, since $G$ is $4$-connected. therefore we are done in this situation.    
Otherwise, for each pair  $\{v_1, v_{2}\}$, $\{v_2, v_{3}\}$, $\{v_3, v_{4}\}$, $\{v_4, v_{1}\}$, there is a distinct vertex which is adjacent to both vertices of the pair. 
Therefore the size of $A$ is at least $4$. 
Hence the size of the vertex set $\cup_{i=1}^{4}A_i$, is at most $x$.  
By the AM-GM inequality, we have that one of $\abs{A_1}\cdot\abs{A_3}$ or $\abs{A_2}\cdot\abs{A_4}$ is at most $\frac{x^2}{16}$.
Therefore we can choose one of the edges  $\{v_1,v_3\}$ or $\{v_2,v_4\}$, such that after adding that edge to the graph $G'$, the Wiener index of the graph decreases by at most $\frac{x^2}{16}$. 
Let us denote the maximal planar graph obtained by adding this edge to $G'$ by $G_{x+4}$. 

Similarly, we denote the maximal planar graph obtained from $G$, by deleting all vertices in the inner region and adding the diagonal which decreases the Wiener index by at most  $\frac{(n-x-4)^2}{16}$ by $G_{n-x}$.  

Consider the graph $G_{n-x}$ and a sub-set of it's vertices $S=\{v_{1},v_{2},v_{3},v_{4}\}$. Since the graph $G$ is $4$-connected, each non-terminal level of $\pa{S}^{G_{n-x}}$ has at least $4$ vertices. Therefore we get that $\sigma_{G_{n-x}}(S)\leq\sigma_4(n-x-4)= \frac{(n-x-2)^2}{8}$, from Lemma \ref{Four_Connected_Lemma}.

Recall that $G'$ is the graph  obtained from $G$ by deleting the vertices from the outer region.
For each $i \in \{1,2,3,4\}$, consider the BFS partition $\pa{v_i}^{G'}$.
Note that, $x\geq 4$, $G$ is $4$-connected, and by  minimality of $x$, $x>1$, we have that every non-terminal level of $\pa{v_i}^{G'}$ has at least $5$ vertices, except for the first level which may contain only four vertices and the level before the last, which may also contain four vertices in this case the last level has size exactly one. Status of the $v_i$ is maximised, if number of vertices in the first and before the last level are four, last level contains only one vertex and every other level contains exactly five vertices. 

 To simplify calculations, of the status of the vertex $v_i$,  we may hang a new temporary vertex on the root and we may bring a vertex from the last level to the previous level.  
 This modifications do not change the status of the vertex, but it increases number of vertices. Now we may apply  Lemma \ref{Five_Connected_Lemma} for this BFS partition considering that number of vertices in all levels is exactly 5.
Therefore we have  $\sigma_{G'}(v_i)\leq \frac{(x+4)^2+5(x+4)}{10}$. Observe that this status contains distances, from $v_i$ to other vertices from the cut set, which equals to four. Note that this is an uniform upper bound for the  status of each of the vertices from the cut set.  

Finally we may upper bound the Wiener index of $G$ in the following way, 
\begin{equation*}
\begin{split}
    W(G) &\leq W(G_{n-x})+\frac{(n-x-4)^2}{16}+W(G_{x+4})+\frac{x^2}{16}-8\\
    &+x\cdot\sigma_{G_{n-x}}(\{v_1,v_2,v_3,v_4\})+(n-x-4)\cdot(\sigma_{G'}(v_1)-4).
\end{split}     
\end{equation*}
 In the first line we upper bound all distances  between pairs of vertices on the cut set and outer region, and   between pairs of vertices on the cut set and inner region. We take minus $8$ since distances between the pairs from the cut set was double counted. 
 In the second line we upper bound all distances from the outer region to the inner region. 
 This distances are split in two, distances from the outer region to the cut set and from the fixed vertex, without loss of generality say $v_1$, of the cycle to the inner region. 
 
We are going to prove that $W(G)\leq \frac{1}{18}(n^3+3n^2)-1$, therefore we will be done in this sub-case. We need to prove the following inequality 
\begin{equation*}
\begin{split}
    \frac{1}{18}(n^3+3n^2)-1 &\geq \frac{1}{18}((n-x)^3+3(n-x)^2)+\frac{(n-x-4)^2}{16}\\
    &+\frac{1}{18}((x+4)^3+3(x+4)^2)+\frac{x^2}{16}-8\\
    &+x\cdot\frac{(n-x-2)^2}{8}+(n-x-4)\cdot(\frac{(x+4)^2+5(x+4)}{10}-4).
\end{split}     
\end{equation*}
After we simplify, we get
\begin{equation}
\begin{split}
\frac{82}{45}- \frac{9n}{10}+\frac{n^2}{16}+ \frac{x}{5} + \frac{41 n x}{120}- \frac{n^2 x}{24}- \frac{3 x^2}{40}+ \frac{n x^2}{60} + \frac{x^3}{40}\leq 0.
\end{split}  
\label{eq1}
\end{equation}
We know that $4\leq x\leq \frac{n-4}{2}$ and if we set $x=4$, we get $2176 + 528 n - 75 n^2\leq 0$ which holds for all $n$, $n\geq 10$. 
Therefore, if the derivative of the right hand side of the inequality is negative for all $\{x\mid 4\leq x\leq \frac{n-4}{2}\}$, then the inequality holds for all these values of $x$.
Differentiating the LHS of the Inequality (\ref{eq1}), with respect to $x$, we get
\begin{align}
\label{oscar1}
\begin{split}
&\frac{\delta}{\delta x} \bigg( \frac{82}{45}- \frac{9n}{10}+\frac{n^2}{16}+ \frac{x}{5} + \frac{41 n x}{120}- \frac{n^2 x}{24}- \frac{3 x^2}{40}+ \frac{n x^2}{60} + \frac{x^3}{40}  \bigg)\\
&=\frac{1}{5} + \frac{41 n}{120} - \frac{n^2}{24} - \frac{3 x}{20} + \frac{n x}{30} + \frac{3 x^2}{40}.
\end{split}     
\end{align}

If we set $x=4$ in Equation \ref{oscar1}, we get  $\frac{1}{120} (96 + 57 n - 5 n^2)$, which is negative for all $n$, $n\geq 13$. 
If we set $x=\frac{n-4}{2}$ in Equation \ref{oscar1}, we get  $\frac{1}{160} (-n^2 + 8 n + 128)$, which is negative for all $n$, $n\geq 17$. Therefore Equation \ref{oscar1} is negative in the whole interval. Since $n\geq 19$, we have $W(G)\leq \frac{1}{18}(n^3+3n^2)-1$, and this sub-case is settled.

\begin{bf} Case $2.2$ \end{bf} In this case, we assume $2\leq x\leq 3$.  
From the minimality of $x$, we have $x=2$. 
Let us consider the maximal planar graph, denoted by $G_{n-2}$, obtained from $G$ by deleting these two vertices from the inner region and adding an edge which decreases the Wiener index by at most $\frac{(n-6)^2}{16}$.  

By the choice of $x$, we have that for a vertex inside the cut set  $v$, each  level of $\pa{v}^{G}$  contains at least $5$ vertices, except the first one which  contains only $4$ and  the level before last may contain $4$ vertices too  followed by one vertex in the last level. Therefore the status of the vertex $v$ is maximized, if the last level contains one vertex, the level before the last and the first level contain four vertices and every other level contains five vertices.
Therefore status of the vertices inside can be bounded by $\sigma_5(n)=\frac{1}{10}(n^2+5n)$. This bound comes from Lemma \ref{Five_Connected_Lemma}, after similar modifications of the BFS partition as in previous case.  
Finally we have,

\begin{equation}
\begin{split}
W(G)&\leq W(G_{n-2})+\frac{(n-6)^2}{16}+\frac{2}{10}(n^2+5n)\\
&\leq  \frac{1}{18}((n-2)^3+3(n-2)^2)+\frac{(n-6)^2}{16}+\frac{2}{10}(n^2+5n)\\
&=\frac{1}{18}n^3 + \frac{23}{240}n^2+\frac{1}{4}n - \frac{89}{36}\leq \frac{1}{18}(n^3+3n^2)-1.
\end{split}     
\end{equation}
The last inequality holds for all $n\geq 10$. Therefore we have settled this sub-case too since $n\geq 19$. 

\begin{bf}Case $2.3$ \end{bf}In this case we assume $x=n-5$. Therefore we have a cut set of size one.  With similar reasoning, as in previous case we get

\begin{equation}
\begin{split}
W(G)&\leq W(G_{n-1})+\frac{(n-5)^2}{16}+\frac{1}{10}(n^2+5n)\\
&\leq  \frac{1}{18}((n-1)^3+3(n-1)^2)+\frac{(n-5)^2}{16}+\frac{1}{10}(n^2+5n)\\
&=\frac{1}{18}n^3 + \frac{13}{80}n^2+\frac{7}{24}n - \frac{241}{144}\leq \frac{1}{18}(n^3+3n^2)-1.
\end{split}     
\end{equation}
The last inequality holds for all $n\geq 9$. 
Therefore, we have settled this sub-case too since $n\geq 19$.

We have considered all sub-cases when $G$ is $4$-connected. 
We proved that in this case Wiener index is strictly less than the desired upper bound.

\noindent\begin{bf} Case $3$. 
Let $G$ be $3$-connected and not $4$-connected. \end{bf} 

Since $G$ is not $4$-connected and it is a maximal planar graph, it must have a cut set of size $3$, say  $\{v_1,v_2,v_3\}$. Which induces a triangle from the Lemma \ref{S_Connected_Cycle_lemma}. 
Let us assume, without loss of generality, that number of vertices in the inner region of the cut set is minimal, say $x$.

\begin{bf} Case $3.1.$  \end{bf} Assume $x\leq 2$. From the minimality of $x$, we have $x \not=2$, therefore $x=1$. 
Let us denote this vertex as $v$. Let $G_{n-1}$ be a maximal planar graph obtained from $G$ by deleting the vertex $v$.
From the Lemma \ref{Three_Connected_Lemma}, we have $\sigma_G(v)\leq \frac{1}{6}(n^2+n)-\frac{1}{3}\mathbb{1}_{3|(n-1)}$. Finally we have, 
\begin{equation}
\begin{split}
    W(G) & \leq W(G_{n-1})+\sigma_G(v) \\
    & \leq \frac{1}{18}((n-1)^3+3(n-1)^2)-\frac{1}{9}\mathbb{1}_{3|n}-\frac{2}{9}\mathbb{1}_{3|(n-2)} \\
    &+\frac{1}{6}(n^2+n)-\frac{1}{3}\mathbb{1}_{3|(n-1)}= \frac{n^3}{18}+\frac{n^2}{6}+\frac{1}{9}-\frac{1}{9}\mathbb{1}_{3|(n)}-\frac{2}{9}\mathbb{1}_{3|(n-2)}-\frac{1}{3}\mathbb{1}_{3|(n-1)}\\
    &\leq \bigg\lfloor\frac{1}{18}(n^3+3n^2)\bigg\rfloor.
\end{split}
    \end{equation}
In this case the equality holds if and only if the graph obtained after deleting the vertex $v$ is $T_{n-1}$. 
We can observe that, if we add the vertex $v$ to the graph $T_{n-1}$, the choice that maximize the status of $v$ is only when we get the graph $T_n$.  Hence we have the desired upper bound of the Wiener index and equality holds if and only if $G=T_n$.

\begin{bf} Case $3.2.$  \end{bf} Assume $x=3$.  
Let us denote vertices in the inner region  as $x_1,\ x_2\text{ and }x_3$. From the minimality of $x$ and maximality of $G$, the structure of $G$ in the inner region is well defined, see Figure  \ref{x=3}.
If we remove these three inner vertices,  the graph we get is denoted by $G_{n-3}$ and is still maximal. 
Hence we may use the induction hypothesis for the graph $G_{n-3}$. 
Consider the graph $G_{n-3}$ and a vertex set $S=\{v_1,v_2,v_3\}$. Each level of $\pa{S}^{G_{n-3}}$ has at least three vertices except the terminal one. 
Therefore we may apply Lemma \ref{Three_Connected_Lemma}, then we have $\sigma_{G_{n-3}}(\{v_1,v_2,v_3\})\leq \frac{1}{6}((n-6)^2+3(n-6)+2)$. 
To estimate distances from the vertices in the outer region to the vertices in the inner region we do the following. We first estimate distances from the outer region to the cut set and from the fixed vertex on the cut set to all $x_i$.
The distances from the vertices in the  outer region to the set $\{v_1,v_2,v_3\}$, is  $\sigma_{G_{n-3}}( \{v_1,v_2,v_3\})$. 
The sum of distances from $v_i$ to the vertices $\{x_1,x_2,x_3\}$ is $4$. 
Note that, if we take a vertex in the outer region which has at least two neighbours on the cut set, then for this vertex we need to count $3$ for the distances from the cut set to the  vertices $\{x_1,x_2,x_3\}$.  
Since we have at least two such vertices, all cross distances can be bounded by $3\sigma_{G_{n-3}}( \{v_1,v_2,v_3\})+4(n-5)+6$. 
Then we have, 
\begin{equation}
\begin{split}
    W(G) & \leq W(G_{n-3})+W(K_3)+3\sigma_{G_{n-3}}(\{v_1,v_2,v_3\})+4n-14 \\
    & \leq \frac{1}{18}((n-3)^3+3(n-3)^2)+\frac{1}{2}((n-6)^2+3(n-6)+2)+4n-11\\
    &< \bigg\lfloor\frac{1}{18}(n^3+3n^2)\bigg\rfloor.
\end{split}
    \end{equation}\
Therefore, this case  is also settled.

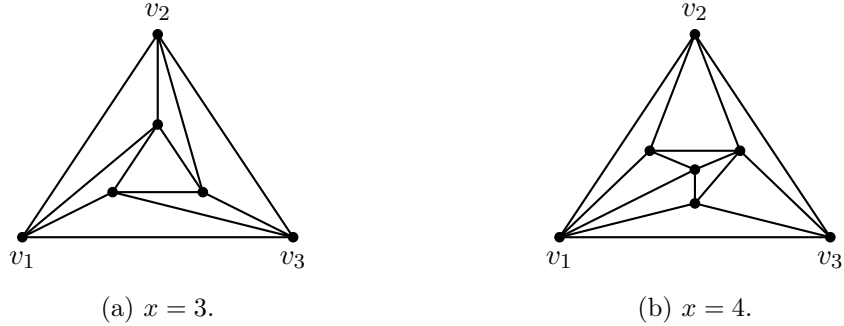
\begin{figure}[h]
\begin{subfigure}{0.55\linewidth}
\centering
\begin{tikzpicture}[scale=0.15]
\draw[fill=black](-12,0)circle(12pt);
\draw[fill=black](12,0)circle(12pt);
\draw[fill=black](0,18)circle(12pt);
\draw[fill=black](-4,4)circle(12pt);
\draw[fill=black](4,4)circle(12pt);
\draw[fill=black](0,10)circle(12pt);
\draw[thick](-12,0)--(0,18)--(12,0)--(-12,0);
\draw[thick](-4,4)--(0,10)--(4,4)--(-4,4);
\draw[thick](-12,0)--(-4,4)(12,0)--(4,4)(0,18)--(0,10)(-12,0)--(0,10)(12,0)--(-4,4)(0,18)--(4,4);
\node at (-12,-2){$v_1$};
\node at (12,-2){$v_3$};
\node at (0,20){$v_2$};
\end{tikzpicture}
\caption{$x=3$.}
\label{x=3}
\end{subfigure}
\begin{subfigure}{0.3\linewidth}
\centering
\begin{tikzpicture}[scale=0.15]
\draw[fill=black](-12,0)circle(12pt);
\draw[fill=black](12,0)circle(12pt);
\draw[fill=black](0,18)circle(12pt);
\draw[fill=black](0,3)circle(12pt);
\draw[fill=black](0,6)circle(12pt);
\draw[fill=black](-4,7.67)circle(12pt);
\draw[fill=black](4,7.67)circle(12pt);
\draw[thick](-12,0)--(12,0)--(0,18)--(-12,0);
\draw[thick](-12,0)--(-4,7.67)--(0,18);
\draw[thick](-12,0)--(0,3)--(12,0);
\draw[thick](12,0)--(4,7.67)--(0,18);
\draw[thick](-12,0)--(0,6)--(4,7.67)--(-4,7.67)--(0,6)--(0,3)--(4,7.67);
\node at (-12,-2) {$v_1$};
\node at (12,-2) {$v_3$};
\node at (0,20) {$v_2$};
\end{tikzpicture}
\caption{$x=4$.}
\label{x=4}
\end{subfigure}
\caption{The unique inner regions for the $3$-connected case when  $x=3$ and $x=4$.}
\label{x=3-4}
\end{figure}

\begin{bf} Case $3.3$  \end{bf} Assume $x=4$. 
From the minimality of $x$ and maximality of the planar graph $G$, the only configuration of the inner region is in Figure \ref{x=4}.
Consider a maximal planar graph on the $n-4$ vertices, say $G_{n-4}$, which is obtained from $G$ by deleting the four inner vertices. 
We will apply the induction hypothesis for this graph, to upper bound the sum of distances between all pairs of vertices from $V(G_{n-4})$ in $G$. 
By applying  Lemma \ref{Three_Connected_Lemma} for $G_{n-4}$ and  $S= \{v_1,v_2,v_3\}$,   we get $\sigma_{G_{n-4}}(\{v_1,v_2,v_3\})\leq\frac{1}{6}((n-4-3)^2+(n-4-3)+2)$. 
The sum of the distances between the four inner vertices is $7$. The sum of the distances from each $v_i$ to all of the vertices inside is at most six.
By following a similar argument as in previous case we have,
\begin{equation}
\begin{split}
    W(G) & \leq \frac{1}{18}((n-4)^3+3(n-4)^2)+7+\frac{4}{6}((n-7)^2+(n-7)+2)+6(n-4)\\
       &< \bigg\lfloor\frac{1}{18}(n^3+3n^2)\bigg\rfloor.
\end{split}
\end{equation}
Therefore, this case is also settled.

\begin{bf} Case $3.4$  \end{bf} Assume $x=5$. 
From the minimality of $x$ and maximality of the planar graph $G$, there are three  configurations of the inner region, see Figure \ref{x=5}. 
Consider a maximal planar graph on the $n-5$ vertices, say $G_{n-5}$, which is  obtained from $G$ by deleting $5$ vertices from the inner region. 
We will  apply the induction hypothesis for this graph $G_{n-5}$, to bound the sum of the distances between the vertices of $V(G_{n-5})$ in the graph $G$. 
 By applying  Lemma \ref{Three_Connected_Lemma} for  $G_{n-5}$ and  $S= \{v_1,v_2,v_3\}$, we get $\sigma_{G_{n-5}}(\{v_1,v_2,v_3\})\leq\frac{1}{6}((n-8)^2+(n-8)+2)$.
 The sum of the distances between five inner vertices is at most $13$. The sum of the distances from $v_i$ to all of the vertices inside is at most $8$.
 Finally  we have,
\begin{equation}
\begin{split}
    W(G) & \leq \frac{1}{18}((n-5)^3+3(n-5)^2)+13+\frac{5}{6}((n-8)^2+(n-8)+2)+8(n-5)\\
         &< \bigg\lfloor\frac{1}{18}(n^3+3n^2)\bigg\rfloor.
\end{split}
\end{equation}
Therefore this case is also settled.

\begin{figure}[h]
\centering
\begin{tikzpicture}[scale=0.15]
\draw[fill=black](-12,0)circle(12pt);
\draw[fill=black](12,0)circle(12pt);
\draw[fill=black](0,18)circle(12pt);
\draw[fill=black](0,4)circle(12pt);
\draw[fill=black](0,2)circle(12pt);
\draw[fill=black](0,6)circle(12pt);
\draw[fill=black](-4,7.67)circle(12pt);
\draw[fill=black](4,7.67)circle(12pt);
\draw[thick](-12,0)--(12,0)--(0,18)--(-12,0);
\draw[thick](-12,0)--(-4,7.67)--(0,18);
\draw[thick](-12,0)--(0,4)--(4,7.67)(0,2)--(4,7.67);
\draw[thick](12,0)--(4,7.67)--(0,18);
\draw[thick](-12,0)--(0,6)--(4,7.67)--(-4,7.67)--(0,6)--(0,4);
\draw[thick](-12,0)--(0,2)--(12,0)(0,2)--(0,4);
\node at (-12,-2){$v_1$};
\node at (12,-2){$v_3$};
\node at (0,20){$v_2$};
\end{tikzpicture}\qquad
\begin{tikzpicture}[scale=0.15]
\draw[fill=black](-12,0)circle(12pt);
\draw[fill=black](12,0)circle(12pt);
\draw[fill=black](0,18)circle(12pt);
\draw[fill=black](-4,4)circle(12pt);
\draw[fill=black](0,3)circle(12pt);
\draw[fill=black](0,9)circle(12pt);
\draw[fill=black](-4,7.67)circle(12pt);
\draw[fill=black](4,7.67)circle(12pt);
\draw[thick](-12,0)--(12,0)--(0,18)--(-12,0);
\draw[thick](-12,0)--(-4,7.67)--(0,18);
\draw[thick](-12,0)--(0,3)--(12,0);
\draw[thick](12,0)--(4,7.67)--(0,18);
\draw[thick](-4,7.67)--(-4,4)--(0,3)--(4,7.67)--(0,9)--(-4,7.67)(0,18)--(0,9);
\draw[thick](-12,0)--(-4,4)--(0,9)--(0,3);
\node at (-12,-2){$v_1$};
\node at (12,-2){$v_3$};
\node at (0,20){$v_2$};
\end{tikzpicture}\qquad
\begin{tikzpicture}[scale=0.15]
\draw[fill=black](-12,0)circle(12pt);
\draw[fill=black](12,0)circle(12pt);
\draw[fill=black](0,18)circle(12pt);
\draw[fill=black](0,3)circle(12pt);
\draw[fill=black](0,6)circle(12pt);
\draw[fill=black](-4,4)circle(12pt);
\draw[fill=black](-4,7.67)circle(12pt);
\draw[fill=black](4,7.67)circle(12pt);
\draw[thick](-12,0)--(12,0)--(0,18)--(-12,0);
\draw[thick](-12,0)--(-4,7.67)--(0,18);
\draw[thick](4,7.67)(0,3)--(4,7.67);
\draw[thick](12,0)--(4,7.67)--(0,18)(-12,0)--(0,3)--(12,0);
\draw[thick](-12,0)--(0,6)--(4,7.67)--(-4,7.67)(0,3)--(0,6)(0,3)--(-4,4);
\draw[thick](-4,4)--(-4,7.67)--(0,6);
\node at (-12,-2){$v_1$};
\node at (12,-2){$v_3$};
\node at (0,20){$v_2$};
\end{tikzpicture}
\caption{3-connected, $x=5$.}
\label{x=5}
\end{figure}
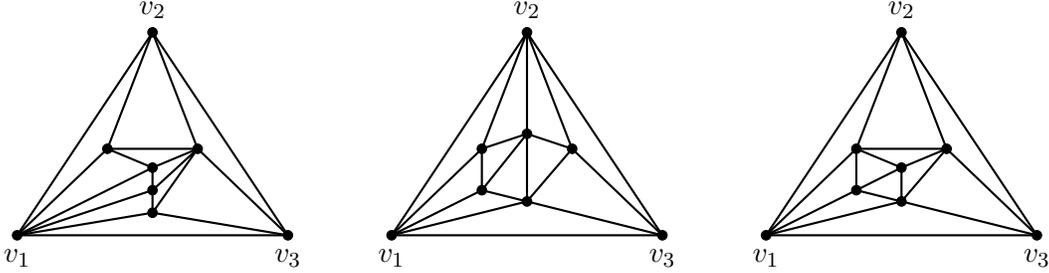

\begin{bf} Case $3.5$  \end{bf} Assume $x\geq 6$. 
First we settle for $x\geq 7$ and then for $x=6$.
Consider the maximal planar graph on $n-x$ vertices, say $G_{n-x}$, which is obtained from $G$ by deleting those $x$ vertices from the inner region of the cut set $\{v_1,v_2,v_3\}$.
Consider the maximal planar graph on $x+3$ vertices, say $G_{x+3}$, which is obtained from  $G$, by deleting all $n-x-3$ vertices from the outer region of the cut set $\{v_1,v_2,v_3\}$.
We know by induction that $W(G_{x+3})\leq \frac{1}{18}((x+3)^3+3(x+3)^2)$. 
There are at least two vertices from the cut set $\{v_1,v_2,v_3\}$, such that each of them has at least two neighbours in the outer region of the cut set. 
Without loss of generality, we may assume they are $v_1$ and $v_2$. 
 Hence if we consider $\pa{v_1}^{G_{n-x}}$ and $\pa{v_2}^{G_{n-x}}$, we will have $4$ vertices in the first level of  and at least three in the following levels until the last one. 
 Therefore  we have $\sigma_{G_{n-x}}(v_1)\leq \sigma_3(n-x-2)+1\leq\frac{1}{6}((n-x-2)^2+3(n-x-2)+8)$  from Lemma \ref{Three_Connected_Lemma} and same for $v_2$.
Now let us consider  $\pa{\{v_1,v_2\}}^{G_{x+3}}$, from minimality of $x$, each non-terminal level of the $\pa{\{v_1,v_2\}}^{G_{x+3}}$ contains at least 4 vertices. 
Therefore by applying Lemma \ref{Four_Connected_Lemma}, we get $\sigma_{G_{x+3}}(\{v_1,v_2\})\leq \frac{1}{8}(x^2+6x+9)$. 
We have,
\begin{equation}
\begin{split}
    W(G) \leq & (W(G_{x+3})+W(G_{n-x})-3)+(n-x-3)(\sigma_{G_{x+3}}(\{v_1,v_2\})-1)\\&
    +x \bigg(\text{max} \bigg\{ \sigma_{G_{n-x}}(v_{1}),\sigma_{G_{n-x}}(v_{2}) \bigg\}-2\bigg).
\end{split}
\end{equation}
The first term of the sum is an upper bound for the sum of all distances which does not cross the cut set.
The second and the third terms upper-bounds all cross distances in the following way- we may split this sum into two parts for each crossing pair sum from inside to $\{v_1, v_2\}$ set and from $v_i$, $i \in \{1,2\}$ to the vertex outside, those are the second and the third terms of the sum accordingly. 
Therefore applying estimates, we get
\begin{equation}\label{nn}
    \begin{split}
    \frac{1}{18}(n^3+3n^2)-1 &\geq  \frac{1}{18}((x+3)^3+3(x+3)^2)+\frac{1}{18}((n-x)^3+3(n-x)^2)-3\\
    & +\frac{(n-x-3)(x^2+6x+1)}{8}+\frac{x((n-x-2)^2+3(n-x-2)-4)}{6}.
    \end{split}
\end{equation}
After simplification we  have
\begin{equation}\label{Tempo1}
-x^3+x^2(n+3)+x(21-6n)-(15+3n)\geq 0.
\end{equation}
where 
\[
\frac{\delta}{\delta x} \bigg(-x^3+x^2(n+3)+x(21-6n)-(15+3n)\bigg)=-3x^2+(2n+6)x+21-6n.
\]
The derivative is positive when  $x\in [7,\frac{n}{2}]$. 
Hence since the inequality (\ref{Tempo1}) holds for  $x=7$, it also holds for all $x$, $x \in [7,\frac{n}{2}]$. 
Therefore, if $x\geq 7$ we are done.


Finally if $x=6$, then distances from $v_1$ and $v_2$ to all vertices inside is $9$ instead of $\frac{73}{8}$ as in \ref{nn}. 
Thus we get an improvement of Inequality (\ref{nn}), which shows that $W(G) < \floor{\frac{1}{18}(n^3+3n^2)}$ even for $x=6$. Therefore we have settled $3$-connected case too. 

\end{proof}
\section{Concluding Remarks}
There is the unique maximal planar graph $T_n$, maximizing the Wiener index, Theorem \ref{Main_Theorem}.
Clearly $T_n$  is not $4$-connected. 
One may ask for the maximum Wiener index for the family of $4$-connected and $5$-connected maximal planar graphs. 
In \cite{29}, 
asymptotic results were proved for both cases. 
Moreover, based on their constructions, they conjecture  sharp bounds for both $4$-connected and $5$-connected maximal planar graphs. 
Their conjectures are the following.
\begin{conjecture}
Let $G$ be an $n$, $n\geq 6$, vertex, maximal, $4$-connected, planar graph. Then we have
\begin{equation*}W(G)\leq 
\begin{cases}
\frac{1}{24}n^3+\frac{1}{4}n^2+\frac{1}{3}n-2, &\text{if $n\equiv 0,2 \; (mod \ 4)$;}\\
\frac{1}{24}n^3+\frac{1}{4}n^2+\frac{5}{24}n-\frac{3}{2}, &\text{if $n\equiv 1 \;(mod \ 4)$.}\\
\frac{1}{24}n^3+\frac{1}{4}n^2+\frac{5}{24}n-1, &\text{if $n\equiv 3 \;(mod \ 4)$;}
\end{cases}
\end{equation*}
\end{conjecture}
\begin{conjecture}
Let $G$ be an $n$, $n\geq 12$, vertex, maximal, $4$-connected, planar graph. Then we have
\begin{equation*}W(G)\leq 
\begin{cases}
\frac{1}{30}n^3+\frac{3}{10}n^2-\frac{23}{15}n+32, &\text{if $n\equiv 0 \;(mod \ 5)$;}\\
\frac{1}{30}n^3+\frac{3}{10}n^2-\frac{23}{15}n+\frac{156}{5}, &\text{if $n\equiv 1 \;(mod \ 5)$.}\\
\frac{1}{30}n^3+\frac{3}{10}n^2-\frac{23}{15}n+\frac{168}{5}, &\text{if $n\equiv 2 \;(mod \ 5)$;}\\
\frac{1}{30}n^3+\frac{3}{10}n^2-\frac{23}{15}n+31, &\text{if $n\equiv 3 \;(mod \ 5)$;}\\
\frac{1}{30}n^3+\frac{3}{10}n^2-\frac{23}{15}n+\frac{161}{5}, &\text{if $n\equiv 4 \;(mod \ 5)$;}\\
\end{cases}
\end{equation*}
\end{conjecture}
\section*{Acknowledgements}
The research of the second and the fourth authors is partially supported by the National Research, Development and Innovation Office -- NKFIH, grant K 116769 and SNN 117879. 
The research of the fourth author is partially supported by  Shota Rustaveli National Science Foundation of Georgia SRNSFG, grant number FR-18-2499.

\end{document}